\documentclass[12pt]{amsart}
\usepackage{amsmath,amssymb}
\usepackage{amsthm}

\theoremstyle{plain}%note that this is the default style (to learn)
\newtheorem{theorem}{Theorem}[section]

\newcounter{tmp}

\newtheorem{proposition}[theorem]{Proposition}
\newtheorem{corollary}[theorem]{Corollary}
\newtheorem{lemma}[theorem]{Lemma}
\theoremstyle{definition}


\begin{document}
\title{Mityagin's Extension Problem.  Progress Report}

\author{Alexander Goncharov and Zel\.iha Ural}


\date{}
\footnote { Alexander Goncharov, Zeliha Ural (Bilkent University, 06800, Ankara, Turkey)\\
e-mail:   goncha@fen.bilkent.edu.tr, zeliha.ural@bilkent.edu.tr}

\maketitle

\begin{abstract} Given a compact set $K\subset {\Bbb R}^d,$ let ${\mathcal E}(K)$ denote
the space of Whitney jets on $K$. The compact set $K$ is said
to have the extension property if there exists a continuous linear
extension operator  $W:{\mathcal E}(K) \longrightarrow
C^{\infty}({\Bbb R}^d)$. In 1961 B. S. Mityagin posed a problem to give
a characterization of the extension property in geometric terms.
We show that there is no such complete description in terms of  densities of Hausdorff
contents or related characteristics. Also the extension property
cannot be characterized in terms of growth of Markov's factors for the set.

\end{abstract}

\section{introduction}

By the celebrated Whitney theorem \cite{wit1}, for each compact set $K\subset {\Bbb R}^d,$
by means of a continuous linear operator one can extend jets of finite order
from  ${\mathcal E}^p(K)$ to functions defined on the whole space,
preserving the order of differentiability.
In the case $p=\infty,$ the possibility of such extension crucially depends on
geometry of the set. Following \cite{tid1}, let us say that $K$  has the {\it extension property (EP)}
if there exists a linear continuous extension operator  $W:{\mathcal E}(K) \longrightarrow
C^{\infty}({\Bbb R}^d)$. For example, any set $K$ with an isolated point does not have $EP$,
 since here each neighborhood of the space ${\mathcal E}(K)$ contains a
linear subspace, but this is not the case for $C^{\infty}({\Bbb R}^d)$.

B. S. Mityagin posed in 1961 (\cite{mit}, p.124) the following problem (in our terms):\\
{\it What is a geometric characterization of the extension property?}

We show that there is no complete characterization of that kind in terms of densities of Hausdorff
contents of sets or analogous functions related to Hausdorff  measures.

This is similar to the state in Potential Theory where R. Nevanlinna \cite{nev}  and H. Ursell \cite{ursel}
proved that there is no complete characterization of polarity of compact sets in terms of
Hausdorff measures. The scale of growth rate of functions $h$, which define the Hausdorff measure $\Lambda_h$,
can be decomposed into three zones. For $h$ from the first zone of small growth, if $0<\Lambda_h(K)$
then the set $K$ is not polar.  For $h$ from the zone of fast growth, if $\Lambda_h(K)<\infty$
then the set $K$ is polar. But between them there is a zone of uncertainty. It is possible to take
two functions with $h_2\prec h_1$ from this zone and the corresponding Cantor-type sets $K_j$ with $0<\Lambda_{h_j}(K_j)<\infty$ for $j\in\{1,2\},$ such that the large (with respect to the Hausdorff measure)
set $K_2$ is polar, whereas the smaller $K_1$ is not polar.

Here we present a similar example of two Cantor-type sets: the smaller set has $EP$ whereas
the larger set does not have it.

Of course, such global characteristics as  Hausdorff  measures or Hausdorff contents cannot be used,
in general, to distinguish $EP,$ which we observe if a compact set is not ``very small" near each its point.
One can suggest for this reason to characterize $EP$ in terms of lower densities of Hausdorff
contents of sets, because, clearly, densities of Hausdorff measures cannot be used for this aim. We analyze
a wide class of dimension functions and show that lower densities of Hausdorff contents do not
distinguish $EP$.

Neither  $EP$ can be characterized in terms of growth rate of Markov's
factors $(M_n(\cdot))_{n=1}^{\infty}$ for sets. Two sets are presented, $K_1$ with $EP$ and
$K_2$ without it, such that $M_n(K_1)$ grows essentially faster than $M_n(K_2)$ as $n\to \infty$.
It should be noted that, by W. Ple\'sniak \cite{ples}, any Markov compact set (with a polynomial growth
rate of $M_n(\cdot)$) has $EP$. All examples are given in terms of the sets $K(\gamma)$ introduced in \cite{G}.\\

Our paper is organized as follows. Section 2 is a short review of main methods of extension.
Also we consider there the Tidten-Vogt linear topological characterization of $EP$.
In Section 3 we give some auxiliary results about the weakly equilibrium Cantor-type set
$K(\gamma).$  In Section 4 we use local Newton interpolations to construct an extension operator $W$.
Sections 5 contains the main result, namely a characterization of $EP$ for  ${\mathcal E}(K(\gamma))$
in terms of a sequence related to $\gamma.$ In sections 6 we compare $W$ with the extension operator from 
\cite{GU}, which is given by individual extensions of elements of Schauder basis for the space  ${\mathcal E}(K(\gamma))$.
In Section 7 we consider two examples that correspond to regular and irregular behaviour of
the sequence $\gamma$. In Section 8 we calculate the Hausdorff $h-$measure of $K(\gamma)$ for
a siutable dimension function $h$ and present  Ursell's type example  for $EP.$ In Section 9 we
consider Hausdorff contents and related characteristics.
In Section 10 we compare the growth of Markov's factors  and $EP$ for $K(\gamma)$.\\

 For the basic facts about the spaces of Whitney functions defined on closed subsets of ${\Bbb R}^d$
 see e.g. \cite{fre}, the concepts of the theory of logarithmic potential can be found in \cite{rans}. Throughout
 the paper, $\log$  denotes the natural logarithm. Given compact set $K$, $Cap(K)$ stands for the
 logarithmic capacity of $K$, $Rob(K)=\log(1/Cap(K))\leq \infty$ is the Robin constant for $K$.
 If $K$ is not polar then $\mu_K$ is its equilibrium measure.
 For each set $A$, let  $\#(A)$ be the cardinality of $A$, $|A|$ be the diameter of $A$.
 Also, $[a]$ is the greatest integer in $a$, $\sum_{k=m}^n(\cdots) =0$ and $\prod_{k=m}^n(\cdots) = 1$ if $m>n.$
 The symbol $\sim$ denotes the strong equivalence: $a_n \sim b_n$ means that $a_n=b_n(1+o(1))$ for $n \to \infty.$

\section{Three methods of extension}

Let $K\subset {\Bbb R}^d$ be a compact set, $\alpha =  (\alpha_j)_{j=1}^d \in {\Bbb Z}_+^d$
be a multi-index.  Let $I$ be a closed cube containing $K$ and
 ${\mathcal F}(K,I)=\{F\in C^{\infty}(I): F^{(\alpha)}\large|_K=0, \,\, \forall \alpha\}$
be the ideal of flat on $K$ functions.  The Whitney space  ${\mathcal E}(K)$ of extendable
jets consists of traces on $K$ of $C^{\infty}$-functions defined on $I$, so it is a factor space
of $ C^{\infty}(I)$ and the restriction operator $R: C^{\infty}(I) \longrightarrow {\mathcal E}(K)$
 is surjective. This means that the sequence
 $ 0 \longrightarrow {\mathcal F}(K,I)  \stackrel{J}{\longrightarrow}
C^{\infty}(I) \stackrel{R}{\longrightarrow }{\mathcal E}(K) \longrightarrow 0 $
is exact. If it splits then the right inverse to $R$ is the desired linear continuous extension
operator $W$ and $K$ has $EP$. We see that there always exists a linear extension operator (for example
one can individually extend the elements of a vector basis in ${\mathcal E}(K)$) and a continuous extension
operator, by Whitney's construction. Numerous examples show that a set $K$ has  $EP$
if $K$  is not ``very small" near each its point, but the exact geometric meaning of ``smallness"
has not been comprehended yet.\\

In \cite{tid1} M. Tidten applied D. Vogt's theory of splitting of short exact sequences of Fr\'{e}chet spaces
(see e.g \cite{f.a}, Chapter 30) and presented the following important linear topological characterization of $EP$:
{\it a compact set $ K $ has the extension property if and only if the space ${\mathcal E}(K)$
has a dominating norm (satisfies the condition (DN)).}

Recall that a  Fr\'{e}chet space $X$ with an increasing system of seminorms
$(\,|| \cdot ||_k)_{k=0}^{\infty}$  has a {\it dominating norm} $||\cdot ||_p$ if
for each $q\in \Bbb N$ there exist $r\in \Bbb N$ and $C\geq 1$ such that
$ ||\, \cdot \,||^2_{q} \,\leq C\, ||\cdot ||_{p} \, ||\cdot ||_{r}.$\\

Concerning the question ``How to construct an operator $W$ if it exists?", we can select three main methods
that can be applied for wide families of compact sets.
\\

The first method goes back to  B. S. Mityagin \cite{mit}:
to extend individually the elements $(e_n)_{n=1}^{\infty}$  of a topological basis of ${\mathcal E}(K)$.
Then for $f= \sum_{n=1}^{\infty} \xi_n \cdot e_n$ take $ W(f)= \sum_{n=1}^{\infty} \xi_n \cdot W(e_n).$
See Theorem 2.4 in \cite{vog1} about possibility of suitable simultaneous extensions of $e_n$ in  the case
when $K$ has nonempty interior. The main problem with this method is that we do not know whether each space
 ${\mathcal E}(K)$ has  a topological basis, even though  ${\mathcal E}(K)$ is complemented in
$C^{\infty}(I)$. This is a particular case of the significant
Mityagin-Pe{\l}czy{\'n}ski problem: Suppose $X$ is a nuclear Fr{\'e}chet space with basis and $E$ is a
complemented subspace of $X$. Does $E$ possess a basis? The space $X=s$ of rapidly decreasing sequences,
which is isomorphic to $C^{\infty}(I)$, presents the most important unsolved case.\\

The second method was suggested in \cite{plp}, where  W. Paw{\l}ucki and W. Ple\'{s}niak constructed
an extension operator $W$ in the form of a telescoping series containing Lagrange interpolation polynomials
with Fekete nodes. The authors considered the family of compact sets with polynomial cusps,
but later, in \cite{ples}, the result was generalized to any Markov sets. In fact (see T.3.3 in \cite{ples}),
for each $C^{\infty}$ determining compact set $K$, the operator $W$ is continuous in the so-called
Jackson topology  $\tau_J$ if and only if $\tau_J$ coincides with the natural topology  $\tau$ of the space
${\mathcal E}(K)$ and this happens if and only if the set $K$ is Markov.
We remark that $\tau_J$ is not stronger then  $\tau$ and that  $\tau_J$ always has the dominating norm
property, see e.g. \cite{AG}. Thus, in the case of non-Markov compact set with $EP$ (\cite{G96}, \cite{AG}),
the Paw{\l}ucki-Ple\'{s}niak extension operator is not continuous in $\tau_J$, but this does not
exclude the possibility for it to be bounded in $\tau$. At least for some non-Markov compact sets,
 the local version of this operator is bounded in  $\tau$ (\cite{AG}).\\

In \cite{FJW} L. Frerick, E. Jord{\'a}, and J. Wengenroth showed that, provided some conditions,
 the classical Whitney extension
operator for the space of jets of finite order can be generalized to the case ${\mathcal E}(K).$
Instead of Taylor’s polynomials in the Whitney construction, the authors used a kind of interpolation
by means of certain local measures. A linear tame extension operator was presented for ${\mathcal E}(K)$,
provided $K$ satisfies a local form of Markov's inequality.\\

There are some other methods to construct $W$ for closed sets, for example Seeley's extension \cite{Se}
from a half space or Stein's extension (\cite{St.}, Ch 6) from sets with the Lipschitz boundary. However
these methods, in order to define $W(f,x)$ at some point $x$, essentially require existence of a
line through $x$ with a ray where $f$ is defined, so these methods cannot be applied for compact sets.\\

Here we consider rather small Cantor-type sets that are neither Markov no local Markov.
We follow \cite{AG} in our construction, so  $W$ is a local version of the Paw{\l}ucki-Ple\'{s}niak operator.
It is interesting that, at least for small sets, $W$ can be considered as an operator 
extending basis elements of the space. Thus, for such sets, the first method and a local
version of the second method coincide.
%%%%%
\section{Notations and auxiliary results}

In what follows we will consider only perfect compact sets $K\subset I=[0,1],$
so the Fr{\'e}chet topology $\tau$ in the space $ {\mathcal E} (K)$
can be given by the norms
$$ \|\,f\,\|_{q} = |f|_{q,K} + sup \left\{ \frac { |(R_{y}^{q} f)^{(k)}(x)|}
{|x-y|^{q-k}} : x,y \in K ,x \neq y, k = 0,1,...q \right\} $$
for $q\in {\Bbb Z}_+, $   where $|f|_{q,K} = sup \{ |f^{(k)} (x)| : x \in K
, k \leq q \} \text{ and } R_{y}^{q}f(x) = f(x) - T_{y}^{q} f(x) $
is the Taylor remainder.

Given $f\in  {\mathcal E} (K),$ let $|||\,f\,|||_{\,q} = inf \,\, |\,F\,|_{\,q,I},$  where the infimum
is taken over all possible extensions of $f$ to $F\in C^{\infty}(I).$ By the Lagrange form of the Taylor
remainder, we have $||\,f\,||_{\,q}\leq 3\,|\,F\,|_{\,q,I}$ for any extension $F$.  
The quotient topology $\tau_Q$, given by the norms $(|||\cdot|||_{\,q=0}^{\infty}),$ is complete and,
by the open mapping theorem, is equivalent to $\tau.$ Hence for any $q$ there exist $r\in {\Bbb N}, \,C>0$
such that
 \begin{equation}\label{r}
 |||\,f\,|||_{\,q} \leq C\,||\,f\,||_{\,r}
\end{equation}
for any $f\in{\mathcal E}(K)$. In general, extensions $F$ that realize $|||\,f\,|||_{\,q}$ for a 
given function $f$, essentially depend on $q$. Of course, the extension property of $K$ means the
existence of a simultaneous extension which is suitable for all norms.

Our main subject is the set $K(\gamma)$ introduced in \cite{G}. For the convenience of the reader we repeat
the relevant material.
Given sequence $\gamma = (\gamma_s)_{s=1}^{\infty}$ with $0<\gamma_s < 1/4,$ let $r_0=1$ and $r_s=\gamma_s r_{s-1}^2$ for $s\in \mathbb{N}$. Define $P_2(x)=x(x-1),\, P_{2^{s+1}}=P_{2^s}(P_{2^s}+r_s)$ and
$ E_s=\{x\in {\mathbb{R}}:\,P_{2^{s+1}} (x)\leq 0\}$ for $s\in \mathbb{N}.$
Then $E_s=\cup_{j=1}^{2^s}I_{j,s},$ where the $s$-th level {\it basic intervals} $I_{j,s}$  are disjoint and
$\max_{1\leq j \leq 2^s} |I_{j,s}| \to 0$ as $s\to \infty.$ Here,
$(P_{2^s}+r_s/2)(E_s)=[-r_s/2,r_s/2],$ so the sets $E_s$ are polynomial inverse images of intervals.
Since $E_{s+1} \subset E_s$, we have a Cantor-type set  $K(\gamma) := \cap_{s=0}^{\infty} E_s.$

In what follows we will consider only $\gamma$ satisfying the assumptions
\begin{equation}\label{ll}
\gamma_k \leq 1/32\,\,\,\,\,\,\mbox{for}\,\,\,\, \,\,\,k\in {\Bbb N} \,\,\,\,\,\,
\mbox{and} \,\,\,\,\sum_{k=1}^{\infty} \gamma_k <\infty.
\end{equation}

The lengths $l_{j,s}$ of the intervals $I_{j,s}$ of the $s-$th level
are not the same, but, provided \eqref{ll}, we can estimate them in terms of the parameter
$\delta_s=\gamma_1\gamma_2 \cdots \gamma_s$  (\cite{G}, L.6):
\begin{equation}\label{delta}
 \delta_s < l_{j,s} < C_0\,\delta_s \,\,\,\,\,\mbox {for}
\,\,\,\,\,\,\, 1\leq j\leq 2^s,
\end{equation}
where $C_0= \exp( 16\,\sum_{k=1}^{\infty} \gamma_k).$ Each $I_{j,s}$ contains two
{\it adjacent} basic subintervals $I_{2j-1,s+1}$ and $I_{2j,s+1}$. Let
$h_{j,s}=l_{j,s}-l_{2j-1,s+1} - l_{2j,s+1}$ be the distance between them.
By Lemma 4 in \cite{G}, $h_{j,s}> (1-4 \gamma_{s+1})l_{j,s}.$ Therefore,
\begin{equation}\label{hh}
h_{j,s}\geq 7/8 \cdot l_{j,s}> 7/8 \cdot \delta_s \,\,\,\mbox{ for all}\,\, j.
\end{equation}

In addition, by T.1 in \cite{G}, the level domains $D_s=\{z\in {\Bbb C}:\, |P_{2^s}(z)+r_s/2|<r_s/2\}$
form a nested family and $K(\gamma)=  \cap_{s=0}^{\infty} \overline D_s.$ The value
$R_s= 2^{-s} \log 2 + \sum_{k=1}^s 2^{-k}\log\frac{1}{\gamma_k}$ represents the Robin constant
of $\overline D_s.$ Therefore, the set $K(\gamma)$ is non-polar if and only if
$Rob(K(\gamma))=  \sum_{n=1}^{\infty} 2^{-n}\log{\frac{1}{\gamma_n}}=
 \sum_{n=1}^{\infty} 2^{-n-1}\log{\frac{1}{\delta_n}} < \infty.$\\

We decompose all zeros of $P_{2^s}$ into $s$ groups. Let
$X_0=\{x_1,x_2\}=\{0,1\}, X_1=\{x_3,x_4\}=\{l_{1,1},1-l_{2,1}\}, \cdots,
X_k= \{l_{1,k},l_{1,k-1}-l_{2,k}, \cdots,
1-l_{2^k,k}\}$ for $k\leq s-1.$ Thus, $X_k=\{x:\, P_{2^k}(x)+r_k=0\}$ contains all
zeros of $P_{2^{k+1}}$ that are not zeros of $P_{2^k}.$
 Set $Y_s=\cup_{k=0}^s X_k.$ Then $P_{2^s}(x)=\prod_{x_k\in Y_{s-1}}(x-x_k).$
Clearly, $\#(X_s)=2^s$ for $s\in \Bbb N$ and $\#(Y_s)=2^{s+1}$ for $s\in \Bbb Z_+.$
We refer $s-${\it th type} points to the elements of $X_s.$

The points from $Y_s$ can be ordered using, as in \cite{CA}, the  {\it rule of increase of the type}.
First we take points from $X_0$ and $X_1$ in the ordering given above. The set $X_2=\{x_5,\cdots,x_8\}$
consists of the points of the second type. We take $x_{j+4}$ as the point which is the closest to $x_j$
for $1\leq j \leq 4.$ Here, $x_5=x_1+l_{1,2},\, x_6=x_2-l_{4,2},$ etc. Similarly,
$X_k=\{x_{2^k+1}, \cdots, x_{2^{k+1}}\}$ can be defined by the previous points.
For $1\leq j\leq 2^k,$ the point $x_j$ is an endpoint of a certain basic interval of  $k$-th level.
Let us take $x_{j+2^k}$ as its another endpoint. Thus, $x_{j+2^k}=x_j \pm l_{i,k},$ where the
sign and $i$ are uniquely defined by $j.$ In the same way, any $N$ points
can be chosen on each basic interval. For example, suppose $2^n\leq N<2^{n+1}$ and the points
$(z_k)_{k=1}^N$ are chosen on $I_{j,s}$ by this rule. Then the set includes all $2^n$ zeros
of $P_{2^{s+n}}$ on $I_{j,s}$ (points of the type $\leq s+n-1$) and some $N-2^n$ points of the type $s+n.$\\

We use two technical lemmas from \cite{GU}. We suppose that $\gamma$ satisfies \eqref{ll}.

Let $2^n\leq N<2^{n+1}$ and $Z=(z_k)_{k=1}^{N+1}$ be chosen on a given $I=I_{j,s}$ by the rule
of increase of the type. Write $Z_N=(z_k)_{k=1}^N$ and $C_1=8/7 \cdot (C_0+1).$
For fixed $x\in \Bbb R$ and finite $A=(a_m),$ let $d_k(x, A):=|x-a_{m_k}| \nearrow.$
%%%%%%
\begingroup
\setcounter{tmp}{\value{theorem}}% store current value of theorem counter
\setcounter{theorem}{0} %assign desired value to theorem counter
\renewcommand\thetheorem{\Alph{theorem}}% locally redefine the representation of the theorem counter
\begin{lemma}
For each $x\in \Bbb R$ with $\delta=\mathrm{dist}(x, Z_N)\leq \delta_{s+n}$ and $z\in Z$ we have\\
$\delta_{s+n}\,\prod_{k=2}^N d_k(x, Z_N)\leq C_1^N\,\prod_{k=2}^{N+1} d_k(z,Z).$
\end{lemma}
\endgroup
\setcounter{theorem}{\thetmp} % restore value of theorem counter

In the next lemma we consider the same $N$ and $Z,$ as above, but now we arrange $z_k$ in increasing order.
For $q=2^m-1$ with $m<n$ and $1\leq j \leq N+1-q,$
let $J=\{z_j, \cdots, z_{j+q}\}$ be $2^m$ consecutive points from $Z$. Given $j,$ we consider all possible
chains of strict embeddings of segments of natural numbers:
$[j,j+q]=[a_0,b_0] \subset [a_1, b_1] \subset \cdots \subset [a_{N-q} ,b_{N-q}]=[1,N+1],$ where
$ a_{k} = a_{k-1},\,b_{k}= b_{k-1}+1$ or $ a_{k} = a_{k-1} -1 ,\,b_{k} = b_{k-1}$ for $1\leq k\leq N-q$.
Every chain generates the product $\prod_{k=1}^{N-q} (z_{b_k}-z_{a_k}).$ For fixed $J,$ let $\Pi(J)$ denote the
minimum of these products for all possible chains.

\begingroup
\setcounter{tmp}{\value{theorem}}
\setcounter{theorem}{0}
\renewcommand\thetheorem{B}
\begin{lemma}
For each $J\subset Z$ there exists $\tilde z\in J$ such that $\prod_{k=q+2}^{N+1} d_k(\tilde z,Z)\leq \Pi(J).$
\end{lemma}
\endgroup
\setcounter{theorem}{\thetmp}
%%%%%%
We will characterize $EP$ of $K(\gamma)$ in terms of the values  $B_k=2^{-k-1} \cdot \log \frac{1}{\delta_k}$
that have Potential Theory meaning: $Rob(K(\gamma))=  \sum_{k=1}^{\infty} B_k.$ The main condition is
(compare with (3) in \cite{ep}):
\begin{equation}\label{Y}
\frac{B_{n+s}}{\sum^{n+s}_{k=s} B_k} \rightrightarrows 0 \,\,\,\,\mbox {as}\,\,\,\,  n \rightarrow \infty \,\,\,
\,\,\mbox {uniformly with respect to}\,\,\,\, s.
\end{equation}

We see that this condition allows polar sets.\\

{\bf Example 1}. Let $\gamma_1=\exp(-4 B)$ and $\gamma_k=\exp(-2^k B)$ for $k\geq 2,$ where
$B\geq \frac{1}{4}\log 32,$ so \eqref{ll} is valid. Here, $B_k=B$ for all $k$.
Hence \eqref{Y} is satisfied and the set $K(\gamma)$ is polar.\\

The condition  \eqref{Y} means that
\begin{equation}\label{Y2}
\forall \varepsilon \, \exists s_0,\, \exists n_0:\,\,B_{s+n}< \varepsilon (B_{s}+\cdots +B_{s+n})\,\,\,
\mbox{for}\,\, \,n\geq n_0,\,s\geq s_0.
\end{equation}
Clearly, instead of $\exists s_0$ one can take above $\forall s_0.$
Let us show that  \eqref{Y2} is equivalent to
\begin{equation}\label{Y3}
\forall \varepsilon_1 \,\, \forall m\in \Bbb Z_+ \exists N: \,B_{s+n-m}+\cdots +B_{s+n}
< \varepsilon_1 (B_{s}+\cdots +B_{s+n-m-1}), \,n\geq N, s\geq 1.
\end{equation}
Indeed, the value $m=0$ in \eqref{Y3} gives \eqref{Y2} at once. For the converse,
remark that in \eqref{Y3} we can take on the right side $\varepsilon_1 (B_{s}+\cdots +B_{s+n}),$
so here we consider  \eqref{Y3} in this new form. Suppose \eqref{Y2} is valid.
Given $ \varepsilon_1$ and $m,$ take $\varepsilon=\varepsilon_1/(m+1)$ and the corresponding value $n_0$
from  \eqref{Y2}. Take $N=n_0+m.$ Then for $n\geq N$ and $0\leq k \leq m$ we have $n-k\geq n_0,$ so
$B_{s+n-k}< \varepsilon (B_{s}+\cdots +B_{s+n-k})< \varepsilon (B_{s}+\cdots +B_{s+n}).$
Summing these inequalities, we obtain a new form of \eqref{Y3}.

It follows that the negation of the main condition can be written as
\begin{equation}\label{Not}
\exists \varepsilon \,\, \exists m:  \,  \forall N \,\exists n > N: \,\,\sum_{s+n-m}^{s+n}B_k
> \varepsilon \sum_{s}^{s+n-m-1}B_k\,\, \mbox{for}\,\, s=s_j \uparrow \infty.
\end{equation}

Also, \eqref{Y2} is  equivalent to
\begin{equation}\label{Y4}
\forall \varepsilon\,\, \, \exists m, n_0, s_0: \, B_{s+n}
<  \varepsilon( B_{s+n-m} +\cdots +B_{s+n-1}) \,\,\, \mbox{for}\,\,n\geq n_0,\,s\geq s_0.
\end{equation}
Indeed, comparison of right sides of inequalities shows that \eqref{Y4} implies \eqref{Y2}.
Conversely, given $ \varepsilon,$ take $n_0$ such that \eqref{Y2} is valid with
$\varepsilon/(1+\varepsilon)$ instead of $\varepsilon.$ Take $m=n_0.$ Then for $n\geq n_0, s\geq s_0$
we have $\tilde s=s+n-m\geq s_0$ and, by \eqref{Y2},
$B_{s+n}=B_{\tilde s+m}<\frac{\varepsilon}{1+\varepsilon}(B_{\tilde s}+\cdots B_{\tilde s+m}),$
which is \eqref{Y4}.

We will use a ``geometric" version of \eqref{Y4} in terms of $(\delta_k)$
\begin{equation}\label{G}
\forall M\,\, \, \exists m, n_0, s_0: \,\,
\delta_{s+n-1}\,\delta_{s+n-2}^2 \cdots \delta_{s+n-m}^{2^{m-1}}<\delta_{s+n}^M \,\,\,
\mbox{for}\,\,n\geq n_0,\,s\geq s_0.
\end{equation}

\section{Extension operator for  ${\mathcal E}(K(\gamma))$ }

Here, as in \cite{AG}, we use the method of local Newton interpolations. Let $K$ be shorthand for $K(\gamma)$.
We fix a nondecreasing sequence of natural numbers $(n_s)_{s=0}^{\infty}$ with $ n_s\geq 2$ and
$n_s \to\infty.$ Given function $f$ on $K,$ we interpolate $f$ at $2^{n_0}$ points that are chosen by the rule
of increase of the type on the whole set. A half of points are located on $K\cap I_{1,1}.$ We continue interpolation
on this set up to the degree $2^{n_1}.$ Separately we do the same on $K\cap I_{2,1}.$
Continuing in this fashion, we interpolate $f$ with higher and higher degrees on smaller and smaller
basic intervals. At each step the additional points are chosen by the rule of increase of the type.
Interpolation on $I_{j,s}$ does not affect other intervals of the same level due to the following function.

Let $t>0$ and a compact set $E$ on the line be given. Then $ u(\cdot,t, E)$ is a $C^{\infty}-$
function with the properties: $u(\cdot,t, E) \equiv 1 $ on $E$, \,\,$u(x,t, E) = 0$
for $ dist(x,E) > t$ and
$\sup_{x\in K} |u^{(p)}_{x^p}(x,t,K)| \leq  c_p \,\,t^{-p},$  where the constant $c_p$ depends
only on $p$. Let $c_p\nearrow.$

Given $N+1$ points $(z_k)_{k=1}^{N+1}$ on $K\cap I_{j,s}$  let $L_N(f,x,I_{j,s}) = \sum _{k=1}^{N+1} f(z_k)\,\omega_k(x)$, where $\omega_k(x) =\frac{\Omega_{N+1} (x)}{(x-z_k)\Omega_{N+1}'(z_k)}$ with
$\Omega_{N+1} (x) = \prod_{k=1}^{N+1}(x-z_k)$.

 Let $ N_s= 2^{n_s}-1$ and $ M_s= 2^{n_{s-1}-1}-1$ for $s\geq 1, M_0 =1$. Then, for fixed $s,$
 we take $M_s+1 \leq N \leq N_s,$ so $2^{n}\leq N < 2^{n+1}$ with $n\in\{n_{s-1}-1, \cdots, n_s-1\}.$
For such $N$ and $s$ we take $t_N:=\delta_{s+n}.$  Let, in addition, $1\leq j\leq 2^s$  be fixed.
Then we choose $N+1$ points on the interval $I_{j,s}$ by the rule of increase of the type and consider
for given $f$
$$ A_{N,j,s}:= [L_N(f,x,I_{j,s})-L_{N-1}(f,x,I_{j,s})]\,u(x, t_N,I_{j,s}\cap K).$$

 We call $ A_{j,s}(f,x) :=\sum_{N=M_s+1}^{N_s} A_{N,j,s} $ the  {\it accumulation sum}.
 The last term here corresponds to the interpolation on $I_{j,s}$ at  $2^{n_s}$ points.
 In order to continue interpolation on subintervals of $I_{j,s},$ let us consider
 the  {\it transition sum}
 $$ T_{k,s}(f,x) :=  [L_{M_{s+1}}(f,x,I_{k,s+1})-L_{N_s}(f,x,I_{j, \,s})]
 \,u(x,\delta_{s+n_s-1},I_{k,\,s+1}\cap K),$$
 where we suppose $1\leq k\leq 2^{s+1},\, j=[\frac{k+1}{2}]$ and $ I_{j,\,s} \supset I_{k,s+1} \cup I_{i,s+1}.$

As above, we represent the difference in brackets in the telescoping form:
$$[L_{M_{s+1}}-L_{N_s}]= - \sum_{N=2^{n_s -1}}^{2^{n_s} -1} [L_N(f,x,I_{j, \,s}) -L_{N-1}(f,x,I_{j, \,s})].$$
Here, the interpolating set $Z$ for $L_N$ consists of $M_{s+1}+1$ points
of $Y_{s+n_s-1}\cap I_{k,s+1}$ and $N-M_{s+1}$ points, chosen by the rule of increase of the
type on $I_{i,s+1}.$ The second parameter of $u$ is smaller than the mesh size of $Z$,
so $T_{k,s}(f,x)\ne 0$ only near $I_{k,s+1}.$

Consider a linear operator
$$ W(f,\cdot) = L_{M_0}(f,\cdot,I_{1,\,0})\,u(\cdot, 1, K)+
\sum _{s=0}^{\infty} \big [ \sum_{j=1}^{2^s}  A_{j,s}(f,\cdot) +
\sum_{k=1}^{2^{s+1}}  T_{k,s}(f,\cdot)\big ].$$
We remark at the outset that, for fixed $x\in  \Bbb R$ and $s,$ because of the choice of parameters for
the function $u,$  at most one value $A_{j,s}$ does not vanish. The same is valid for $ T_{k,s}.$

Let us show that $W$ extends functions from ${\mathcal E}(K),$ provided a suitable choice of $(n_s)_{s=0}^{\infty}.$
Define $n_0=n_1=2$ and $n_s=[ \log_2 \log \frac{1}{\delta_s}]$ for $s\geq 2.$ Then $n_s\leq n_{s+1}$ and
\begin{equation}\label{ns}
\frac{1}{2} \log \frac{1}{\delta_s}< 2^{n_s} \leq \log \frac{1}{\delta_s}\,\,\, \mbox {for}\,\,s\geq 2.
\end{equation}

\begin{lemma}\label{ext}
Let $(n_s)_{s=0}^{\infty}$ be given as above. Then for any $\,f\in {\mathcal{E}}(K(\gamma))\,$  and
$\, x \in K(\gamma)\,$ we have $ \,\,W(f,x) = f(x).$
\end{lemma}

\begin{proof}
Let us fix a natural number $q$ with $q > 2+\log(8C_0/7),$ where $C_0$ is defined in \eqref{delta}.
By the telescoping effect,
\begin{equation}\label{tel}
W(f,x) = \lim _{s \to \infty} L_{M_{s}}(f,x,I_{j,s}),
 \end{equation}
where $j =j(s,x)$ is chosen in a such way that $ x \in I_{j,s}.$ As in [EvI],
\begin{equation}\label{LM}
 |\,L_{ M_{s} }(f,x,I_{j,s}) - f(x)| \leq ||\,f\,||_{\,q}
  \,\sum_{k=1}^{2^n}\,|\,x - z_k |^{\,q} \,|\,\omega_k(x)\,|.
  \end{equation}
Here $n$ is shorthand for $n_{s-1}-1$ and $s$ is such that $ M_{s}=2^n-1>q.$ The interpolating set
$(z_k)_{k=1}^{2^n}$ for $L_{ M_{s}}$ consists of all points of the type $\leq s+n-1$ on $I_{j,s}.$
Given point $x,$ we consider the chain of basic intervals containing it:
$x\in I_{j_{n},s+n}\subset  \cdots \subset I_{j_1,s+1} \subset I_{j,s}.$
We see that $I_{j_{n},s+n}$ contains one interpolating point, $I_{j_{n-1},s+n-1}\setminus I_{j_{n},s+n}$
does one more $z_i,$ $I_{j_{n-2},s+n-2}\setminus I_{j_{n-1},s+n-1}$
  contains two such points, etc. Thus, for fixed $k,$ we get

$$ |\,x - z_k |^{\,q}\prod_{i=1, i\ne k}^{2^n}|x-z_i| \leq l_{j,s}^{\,q-1}
 \cdot l_{j_{n},s+n} \cdot l_{j_{n-1},s+n-1}\cdot l_{j_{n-2},s+n-2}^2 \cdots l_{j,s}^{\,2^{\,n-1}}.$$

 By \eqref{delta}, this does not exceed
 $C_0^{2^n+q-1}\delta_{s+n}\,\delta_{s+n-1}\,\delta_{s+n-2}^2 \cdots \delta_{s+1}^{2^{n-2}}\,\delta_{s}^{2^{n-1}+q-1}.$

On the other hand, by a similar argument, for the denominator of $|\,\omega_k(x)\,|$ we have

$$ |\,z_k-z_1|\cdots |\,z_k-z_{k-1}|\,\cdot|\,z_k-z_{k+1}|\cdots|\,z_k-z_{2^n}| \geq
l_{q_{n-1},s+n-1}\cdot h_{q_{n-2},s+n-2}^2 \cdots h_{j,s}^{\,2^{\,n-1}}$$
for some indices $q_{n-1}, q_{n-2}, \cdots.$ The last product exceeds
$(7/8)^{2^n-2} \delta_{s+n-1}\,\delta_{s+n-2}^2 \cdots \delta_{s}^{2^{n-1}},$ by \eqref{hh}.
 It follows that
 $$ \mbox{LHS of }\eqref{LM}\leq
 ||\,f\,||_{\,q}\, 2^n \,C_0^{q-1} (8C_0/7)^{\,2^{\,n}}\,\delta_{s+n}\, \delta_{s}^{q-1} .$$
The expression on the right side approaches zero as $s\to \infty.$ Indeed, $2^n < \log(1/\delta_{s-1}),$
by \eqref{ns}, and $ 2^n (8C_0/7)^{\,2^{\,n}}\, \delta_{s}^{q-1}<1$ due to the choice of $q$.
Thus the limit in \eqref{tel} exists and equals $f(x).$
\end{proof}

\section{Extension property of weakly equilibrium Cantor-type sets}

We need two more lemmas.

\begin{lemma} \label{f} Let $\gamma$  satisfy \eqref{ll}, $q=2^m,\, r=2^n$ with $m<n$
and $Z=(z_k)_{k=1}^r$ be all points of the type $\leq s+n-1$
on $I_{1,s}$ for some $s\in \Bbb Z_+.$ Let $f(x)=\prod^r_{k=1}(x-z_k)$ for $x \in K(\gamma) \cap I_{1,s}$ and $f=0$
on $K(\gamma) \setminus I_{1,s}$. Then $|f|_{0,K(\gamma)}\leq C_0^r \cdot \delta_{n+s} \cdot \delta_{n+s-1} \cdot \delta_{n+s-2}^2 \cdots \delta_{s}^{2^{n-1}}, \,\,|f^{(q)}(0)| \geq q! \cdot (7/8)^{r-q} \cdot  \delta_{n+s-m-1}^{2^m} \cdots \delta_{s}^{2^{n-1}}$ and $|| f ||_r \leq 2 \cdot r!.$
\end{lemma}

\begin{proof}
Fix $\tilde{x}$ that realizes $|f|_{0,K(\gamma)}$ and a chain of basic intervals containing this point:
$\tilde{x} \in I_{j_0,n+s} \subset I_{j_1,n+s-1} \subset \cdots \subset I_{j_n,s}=I_{1,s}$. Arguing as in
Lemma \ref{ext}, we see that
$|f|_{0,K(\gamma)} \leq  l_{j_0,n+s} \cdot l_{j_1,n+s-1} \cdot l_{j_2,n+s-2}^2\cdots  l_{1,s}^{2^{n-1}},$
which, by \eqref{delta}, gives the desired bound.

In order to estimate $|f^{(q)}(0)|,$ let us remark that $f^{(q)}(x)$ is a sum of $\binom{r}{q}$ products, each product has a coefficient $q!$ and consists of $r-q$ terms $(x-z_k)$. One of these products is $g(x):= \prod^{r}_{k=q+1}(x-z_{i_k}),$ where $z_{i_1}< z_{i_2}< \cdots < z_{i_r}.$ All products are nonnegative at $x=0$, since $r-q$ is even. From here, $|f^{(q)}(0)| \geq q! \cdot g(0)$. Taking into account the location of points from $Z$, we get $g(0)= \prod^{r}_{k=q+1} z_{i_k} > h^{2^m}_{1,n+s-m-1} \cdots h^{2^{n-1}}_{1,s}>(7/8)^{r-q} \cdot  \delta_{n+s-m-1}^{2^m} \cdots \delta_{s}^{2^{n-1}},$ by \eqref{hh}. The bound of $\lVert f \rVert_r$ is evident.
\end{proof}

In the next Lemma, for given $2^{n}\leq N < 2^{n+1},$ we consider   $\Omega_N(x)= \prod_{k=1}^{N}(x-z_k)$
with $Z_N=(z_k)_{k=1}^N,$ where the points are chosen on $I_{j,s}$ by the rule of increase of the type. Let
$u(x)=u(x, \delta_{s+n},I_{j,s}\cap K(\gamma))$ and, as above, $d_i(x, Z_N):=|x-z_{k_i}| \nearrow.$

\begin{lemma} \label{om}
The bound
$\,\,|(\Omega_N \cdot u)^{(p)}(x)|\leq 2^p\,(C_0+1)\,c_p\,\delta_{s+n}^{-p+1}\,N^p\,\prod_{k=2}^N d_k(x, Z_N)$

is valid for each $p<N$ and $x\in  \Bbb R.$
\end{lemma}

\begin{proof}
By Leibnitz's formula,
$|(\Omega_N \cdot u)^{(p)}(x)|\leq   \sum_{i=0}^p \binom{p}{i} |\,\Omega_N^{(i)}(x)|\,c_{p-i} \delta_{s+n}^{-p+i}.$
Since $d_k$ increases, we have $|\Omega_N^{(i)}(x)|\leq  \frac{N!}{(N-i)!}\prod_{k=i+1}^N d_k(x, Z_N).$
This gives
\begin{equation}\label{omega}
 |(\Omega_N \cdot u)^{(p)}(x)|\leq 2^p\,c_p\,\delta_{s+n}^{-p}\cdot
\max_{0\leq i\leq p}\,(N\,\delta_{s+n})^i \prod_{k=i+1}^N d_k(x, Z_N).
\end{equation}

The set $Z_N$ consists of $2^n$ endpoints of subintervals of the level $s+n-1$ covered by $I_{j,s}$
and $N-2^n$ points of the type $s+n.$ Here, $\mathrm{dist}(x, I_{j,s}\cap K)=|x-x_0| \leq \delta_{s+n}$
for some $x_0.$ Let $x_0 \in I_{i,s+n} \subset I_{m,s+n-1}.$  Then $I_{m,s+n-1}$
contains from 2 to 4 points of $Z_N.$ In all cases,
$d_1(x, Z_N)\leq l_{i,s+n}+ \delta_{s+n}\leq(C_0+1) \delta_{s+n},$ by \eqref{delta}. Also,
 $ \delta_{s+n}/2 \leq d_2 \leq (C_0+1) \delta_{s+n-1}.$ Here the lower bound corresponds to the case
$\#(I_{i,s+n} \cap Z_N)=2,$  whereas the upper bound deals with $\#(I_{m,s+n-1} \cap Z_N)=2.$
Similarly, $ d_3 \geq h_{m,s+n-1}-\delta_{s+n}.$ From \eqref{hh} and \eqref{ll} it follows that
 $d_3\geq 7/8 \,\,\delta_{s+n-1} -\delta_{s+n}\geq 27 \,\delta_{s+n}.$ This gives
$\delta_{s+n}^{i-1} d_{i+1}\cdots d_N \leq (C_0+1)d_2\cdots d_N$ for $0\leq i\leq p$
and, by \eqref{omega}, the lemma follows.
\end{proof}

We can now formulate our main result.

\begin{theorem}
Suppose $\gamma$ satisfies \eqref{ll}. Then $K(\gamma)$ has the extension property if and only if
 \eqref{Y} is valid.
\end{theorem}

\begin{proof}
Recall that the extension property of a set is equivalent to the condition $(DN)$ of the
corresponding Whitney space. Due to L. Frerick [Fr, Prop. 3.8], ${\mathcal E}(K)$
satisfies $(DN)$ if and only if for any $\varepsilon >0$ and for any $q \in \Bbb N$ there
exist $r \in \Bbb N$ and $C >0$ such that
$ |\cdot |_{\,q}^{1+ \varepsilon}\leq C |\cdot|_{\,0, K}\,\,\,||\cdot||_{\,r}^{\,\varepsilon}.$
Hence, in order to prove that  \eqref{Y} is necessary for $EP$ of $K(\gamma),$
we can show that \eqref{Not} implies the lack of $(DN)$ for ${\mathcal E}(K(\gamma))$,
that is there exist  $\varepsilon>0$ and
$q$ such that for any $r \in \Bbb N$ one can find a sequence
$(f_j) \subset {\mathcal E} (K(\gamma))$ with
$$
|\,f_{j}|_{\,q}^{1+\varepsilon} \,|\,f_{j}|_{\,0, K(\gamma)}^{-1}\,\,||\,f_{j}||_{\,r}^{-\,\varepsilon}\,\,
 \rightarrow \infty \,\,\,\,\,
\mbox{as}\,\,\,\,\, j\rightarrow\infty.
$$
Let us fix $\varepsilon$ and $m$ from the condition \eqref{Not} and take $q=2^m.$
For each fixed large $r$ (clearly, we can take it in the form $r=2^n$) and $s_j$ defined
by \eqref{Not}, we consider the function $f_j$ given in Lemma  \ref{f} for $s=s_j.$ Then
$$C\, |\,f_{j}|_{\,q}^{1+\varepsilon} \,|\,f_{j}|_{\,0, K(\gamma)}^{-1}\,\,||\,f_{j}||_{\,r}^{-\,\varepsilon}
\geq (\delta_{n+s} \cdot \delta_{n+s-1} \cdot \delta_{n+s-2}^2 \cdots \delta_{n+s-m}^{2^{m-1}} )^{-1}
( \delta_{n+s-m-1}^{2^{m}} \cdots   \delta_{s}^{2^{n-1}})^{\varepsilon}, $$
where $C$ does not depend on $j$. The right side here goes to infinity. Indeed, its logarithm
is $2^{n+s}\,\{2 B_{n+s}+  B_{n+s-1}+\cdots +  B_{n+s-m} - \varepsilon [ B_{n+s-m-1}+ \cdots +B_s]\}$
and the expression in braces exceeds $ B_{n+s}$  by \eqref{Not}. Therefore the whole value exceeds
$2^{n+s}B_{n+s}=\frac{1}{2} \log\frac{1}{\delta_{s+n}},$ which goes to infinity when $s=s_j$
increases. Thus,  $EP$ of $K(\gamma)$ implies  \eqref{Y}.\\

For the converse, we consider the extension operator $W$ from Section 5, where $(n_s)$ are chosen as in
\eqref{ns}. We proceed to show that $W$ is bounded provided \eqref{G}.
Let us fix any natural number $p.$ This $p$ and $C_1$ from Lemma A define $M=2p+2+\log (2C_1).$
We fix $m\in  \Bbb N$ that corresponds to $M$ it the sense of \eqref{G}.
Let $q=2^m-1$ and $r=r(q)$ be defined by \eqref{r}. We will show that the bound  $|(W(f,x))^{(p)}|\leq C\,||f||_r$
is valid for some constant $C=C(p)$ and all  $f\in {\mathcal E}(K), \,x \in \Bbb R.$

Given $f$ and $x$, let us consider terms of accumulation sums. For fixed $s\in \Bbb N$ we choose $j\leq 2^s$ such
that $x \in I_{j,s}.$ Fix $N$ with $ 2^n\leq N < 2^{n+1}$ for $n_{s-1}-1\leq n \leq n_s-1,$ so
$M_s+1\leq N \leq N_s.$ For large enough $s$ the value $N$ exceeds $p$ and $q$.
As in Lemma A, let $Z=(z_k)_{k=1}^{N+1}=Z_N\cup\{z_{N+1}\}.$
By Newton's representation of
interpolating operator in terms of divided differences, we have
$$
A_{N,j,s}(f,x)= [z_1,\cdots,z_{N+1}]f\,\cdot \Omega_N(x) \,u(x) ,$$
where  $\Omega_N$ and $u$ are taken as in Lemma \ref{om}. We aim to show that

\begin{equation}\label{an} N_s\,|A_{N,j,s}^{(p)}(f,x)|\leq s^{-2}\,||f||_r
\end{equation}
 for large $s$. This gives convergence of the
accumulation sums.

Since divided differences are symmetric in their arguments, we can use (4) from \cite{AG} :
\begin{equation}\label{DivDif}
|\,\,[z_1,\cdots,z_{N+1}]f\,| \leq 2^{N-\,q}\, |||f|||_{\,q}\,\, ( \Pi(J_0))^{-1},
\end{equation}
where $\Pi(J_0)= min_{1\leq j \leq N+1-q}\Pi(J)$ for $\Pi(J)$ defined in Lemma B.
Fix $\tilde z \in J_0$ that corresponds to this set in the sense of Lemma B.

Applying Lemma \ref{om} and Lemma A for $z=\tilde z$  yields
$$ |(\Omega_N \cdot u)^{(p)}(x)|\leq C \, \delta_{s+n}^{-p} \, N^p\, C_1^N \prod_{k=2}^{N+1} d_k(\tilde z, Z)
\,\,\,\mbox{with}\,\, C=2^p(C_0+1)\,c_p.$$
On the other hand, \eqref{DivDif} and Lemma B for $J_0$ give
$$ |\,\,[z_1,\cdots,z_{N+1}]f\,| \leq 2^{N-\,q}\, |||f|||_{\,q}\,\, \prod_{k=q+2}^{N+1} d_k^{-1}(\tilde z, Z).$$

Combining these we see that
$$ |A_{N,j,s}^{(p)}(f,x)| \leq C \,\, |||f|||_{\,q}\, \delta_{s+n}^{-p} \, N^p\, (2C_1)^N \prod_{k=2}^{q+1} d_k(\tilde z, Z).$$

Recall that the set $Z$ includes all points of the type $\leq s+n-1$ on $I_{j,s}$ and $N-2^n$ points of the
type $s+n.$ We can only enlarge the product $\prod_{k=2}^{q+1} d_k(\tilde z, Z)$ if we will consider
only distances from $\tilde z$ to points from $Y_{s+n-1}\cap I_{j,s}.$ Arguing as in Lemma \ref{ext}, we
get  $\prod_{k=2}^{q+1} d_k(\tilde z, Z)\leq C_0^q \delta_{s+n-1}\,\delta_{s+n-2}^2 \cdots \delta_{s+n-m}^{2^{m-1}}.$ We observe that $d_1(\tilde z, Z)=0$ is not included into the product on the left side. By \eqref{G},
 $\prod_{k=2}^{q+1} d_k(\tilde z, Z)\leq C_0^q \,\,\delta_{s+n}^M.$

In order to get \eqref{an}, it is enough to show that
\begin{equation}\label{last}
s^2\,\,N_s \, N^p\, (2C_1)^N \,\delta_{s+n}^{M-p} \to 0 \,\,\mbox{as}\,\,s\to \infty.
\end{equation}
Here, by \eqref{ns}, $N_s \, N^p < 2^{n_s(p+1)}\leq \log (1/\delta_s)^{p+1}< \delta_s^{-p-1}.$ Also,
$(2C_1)^N < \delta_s^{-\log(2C_1)}.$
Clearly, we can replace $\delta_{s+n}$ in \eqref{last} by $\delta_{s}.$ Then,
because of the choice of $M,$ the product in \eqref{last} does not exceed
$s^2 \, \delta_{s},$ which approaches 0 as $s\to \infty,$ since, by \eqref{ll}, $\delta_s\leq 32^{-s}.$

Similar arguments are used for terms of the transition sums.
\end{proof}

\section{Extension of basis elements}

Bases in the spaces ${\mathcal E}(K(\gamma))$

\section{Two examples}
First we consider regular sequences $(B_k)_{k=1}^{\infty}$.
 Let $\beta_k=(\log B_k) /k.$ We say that $(B_k)_{k=1}^{\infty}$ is
{\it regular} if, for some $k_0$, both sequences $(B_k)_{k=k_0}^{\infty}$ and $(\beta_k)_{k=k_0}^{\infty}$
are monotone. Recall that  $(B_k)_{k=1}^{\infty}$ has {\it subexponential growth} if $\beta_k \to 0$ as
$k\to \infty.$

For example, given $a>1$, let
$\gamma^{(1)}_k=k^{-a},\gamma^{(2)}_k=a^{-k},\gamma^{(3)}_k=\exp(-a^k)$ for large enough $k.$
Then $\gamma^{(j)}$ for $1\leq j \leq 3$ generate regular  $B^{(j)}$ with
$B^{(1)}_k \sim 2^{-k-1}\,a\,k\log k, B^{(2)}_k \sim 2^{-k-2}\,k^2\, \log a, B^{(3)}_k \sim (a/2)^{k+1}/(a-1).$
Here, $\beta^{(1)}_k, \beta^{(2)}_k  \nearrow -\log 2$ and $\beta^{(3)}_k  \to -\log(a/2),$ so  $B^{(j)}$
are not of subexponential growth, except $B^{(3)}$ for $a=2.$ We see that \eqref{Y} is valid in the first
two cases and in the third case with $a\leq 2.$

More generally, \eqref{Y} is valid for each
monotone convergent $(B_k)_{k=1}^{\infty}$. Indeed, if $B_k \searrow B\geq 0,$ then LHS of \eqref{Y} does
not exceed $(n+1)^{-1}.$ If $B_k \nearrow B,$ then we take $s_0$ with $B_s > B/2$ for $s\geq s_0.$
Then $B_s+\cdots B_{s+n} \geq (n+1)B/2$ and  LHS of \eqref{Y} $<2(n+1)^{-1}.$ This covers the case of regular
sequences $(B_k)_{k=1}^{\infty}$ when $\beta_k$  are negative. Let us show that \eqref{Y} is valid as well for divergent regular sequences $(B_k)_{k=1}^{\infty}$
of subexponential growth.

\begin{theorem}  \label{sg}
Let $(B_k)_{k=1}^{\infty}$ be regular with positive values of $\beta_k$. Then \eqref{Y} is valid
if and only if $(B_k)_{k=1}^{\infty}$ has subexponential growth.
\end{theorem}

\begin{proof}
A regular sequence  $(B_k)_{k=1}^{\infty}$ is not of subexponential growth, provided $\beta_k >0,$
in the following three cases:  $\beta_k \nearrow \beta <\infty, \,\beta_k \nearrow \infty$
and  $ \beta_k \searrow \varepsilon_0>0.$ We aim to show that \eqref{Y} is not valid under the circumstances.

In the first case, given $s$ and $n$, let $b=\exp \beta_{s+n}.$ Then $b-1\geq \exp \beta_1-1> \beta_1>0$
and $b\leq \exp \beta.$ Here, $\sum_{k=s}^{s+n}B_k< b^{s+n+1}/(b-1)$ as $B_k=\exp(k \beta_k)\leq b^k$
for such $k$. Therefore, $B_{s+n}/\sum_{k=s}^{s+n}B_k > (b-1)/b> \beta_1/\exp \beta,$ which contradicts
 \eqref{Y}.

If $\beta_k \nearrow \infty$ then, by the same argument, $B_{s+n}/\sum_{k=s}^{s+n}B_k > (b-1)/b>1/2$
for $s\geq s_0,$ where $s_0$ is fixed with $\exp \beta_{s_0}>2.$

Suppose $\beta_k \searrow \varepsilon_0.$  We fix indices $s_1<s_2 < \cdots$ such that the  intervals $I_j$
connecting points $(s_j,\beta_{s_j})$ and $(s_{j+1},\beta_{s_{j+1}})$ form a convex envelope
of the set $(k,\beta_k)$ on the plane. We start from $s_1=\max \{s: \beta_s=\beta_1\}.$ If $s_j$ is chosen,
then we take $s_{j+1}$ with he property: for each $k$ with $s_j\leq k \leq s_{j+1}$ the point $(k,\beta_k)$
is not over $I_j$. At any step we can take the next value so large that the slopes of $I_j$  increases to zero.
In addition, given  $s_j$, we take $s_{j+1}$ such that
\begin{equation}\label{beta}
(4- 2\,s_j/s_{j+1}) \beta_{s_{j+1}}\geq (3-s_j/s_{j+1}) \beta_{s_j},
\end{equation}
which is possible as $\beta_k$ decreases to a positive limit.

For fixed $j,$ we take $s=s_j$ and $s+n=s_{j+1}.$ Let $\tilde \beta_k=ak+b$ with $a=-(\beta_s-\beta_{s+n})/n$
and $b=\beta_s+ (\beta_s-\beta_{s+n})\,s/n$ for $s\leq k \leq s+n,$ so the points $(k, \tilde \beta_k)$
are located just on the interval $I_j$. Also, let $g(x) = ax^2+bx$ and
$\tilde B_k =\exp g(k)=\exp(k\,\tilde \beta_k)$ on $[s, s+n].$ Of course, $\tilde B_s=B_s$ and
$\tilde B_{s+n}=B_{s+n}.$

It is easy to check that the function $g$ increases on this interval. Hence,
$\sum_{k=s}^{s+n}B_k \leq \sum_{k=s}^{s+n}\tilde B_k <\int_s^{s+n}g(x)\,dx+B_{s+n}.$ By integration by parts,
$\int_s^{s+n}g(x)\,dx = g(n+s)\cdot [2a(n+s)+b]^{-1}-g(s)\cdot [2as+b]^{-1}+2a\int_s^{s+n}g(x)(2ax+b)^{-2}dx.$
We neglect the last term, as $a<0$, and the second term, as $2as+b=g'(s)>0.$ Also,
$2a(n+s)+b= (2+s/n)\beta_{s+n}-(1+s/n)\beta_s \geq \beta_s/2\geq \ \varepsilon_0/2,$ by \eqref{beta}. Hence
$\int_s^{s+n}g(x)\,dx < 2\, B_{s+n}/\varepsilon_0$ and
$B_{s+n}/\sum_{k=s}^{s+n}B_k > \varepsilon_0/(2+\varepsilon_0),$ so \eqref{Y} is not valid.\\

We proceed to show that \eqref{Y} is valid for $\beta_k \searrow 0,$ that is in the case of subexponential growth
of  $(B_k)_{k=1}^{\infty}$. Here, for fixed large $s$ and $n,$ we estimate $\sum_{k=s}^{s+n}B_k$ from below.
Let $b=\exp \beta_{s+n}.$ Then $B_k\geq b^k$ for $s\leq k \leq s+n$. Therefore,
$B_{s+n}/\sum_{k=s}^{s+n}B_k \leq \frac{b^{n}(b-1)}{b^{n+1}-1}.$

If $b^n<2$ for the given $s$ and $n$ then $b^{n+1}-1>(n+1)\beta_{s+n}.$ On the other hand,
$\exp \beta_{s+n}-1 <2 \,\beta_{s+n}$ for $\beta_{s+n}<1.$ Thus the fraction above does not exceed $4/(n+1).$

Otherwise, $b^n \geq 2$ and $b^{n}< 2(b^{n+1}-1).$ Here the fraction does not exceed $4\,\beta_{s+n}.$
It follows that  $B_{s+n}/\sum_{k=s}^{s+n}B_k \leq \max\{4/(n+1), 4\,\beta_{n}\},$ which is the desired
conclusion.   \end{proof}

Our next objective is to consider irregular sequences $(B_k)_{k=1}^{\infty}$ (compare with Ex.6 in \cite{mia}).
Given two sequences, $(k_j)_{j=1}^{\infty} \subset \Bbb N$ with $k_{j+1}-k_{j} \nearrow \infty$
 and $(\varepsilon_j)_{j=1}^{\infty}$ with $\varepsilon_j \searrow 0,$ let $\gamma_k =(k+5)^{-2}$ for
 $k \ne k_j$ and $\gamma_{k_j} =(k_j+5)^{-2}\varepsilon_j.$ Then $\gamma$ satisfies \eqref{ll} with
 $\delta_k=(5!/(k+5)!)^2\,\varepsilon_1\varepsilon_2\cdots \varepsilon_j$ for $k_j\leq k < k_{j+1}.$
 Let  $A_j:=\log\frac{1}{\varepsilon_1\varepsilon_2\cdots \varepsilon_j}.$ We will consider only sequences
 with the property
 \begin{equation}\label{irr}
 k_{j+1}^2 \cdot A_j^{-1} \to 0 \,\,\, \mbox{as}\,\,\,j\to \infty.
\end{equation}
 Provided this condition, $ B_k=2^{-k} \log\frac{(k+5)!}{5!}+2^{-k-1}\,A_j \sim 2^{-k-1}\,A_j$ for
 $k_j\leq k < k_{j+1}.$
 In addition, an easy computation shows that for large $j,$
 \begin{equation}\label{sum}
 B_{k_j}+ B_{k_j+1}+\cdots +B_{k_{j+1}-1}< 3\,B_{k_j}.
\end{equation}
 Now we can construct different examples of compact sets $K(\gamma)$ without extension property.

 {\bf Example 2}. Let $A_j=2^{k_j},$ so $\varepsilon_j=\exp(-2^{k_j}+ 2^{k_{j-1}})$ for $j\geq 2$ and
$\varepsilon_1=\exp(-2^{k_1}).$ In this case, \eqref{irr} is valid under mild restriction
$2^{-k_j}\,k_{j+1}^2 \to 0$ as $j\to \infty.$ Let us take $s=k_j, n= k_{j+1}-k_j.$ Then $B_{s+n}>2^{-k_{j+1}-1}\,A_{j+1}=1/2$
and, by \eqref{sum}, $B_s+\cdots + B_{s+n-1}<3\,B_s < 4 \cdot 2^{-k_j-1}\,A_{j}=2.$ This gives
\eqref{Not} with $\varepsilon=1/4$ and $m=0.$

 \section{Extension Property of $K(\gamma)$ and Hausdorff measures}

From now on, $h$ is a {\it dimension function}, which means that $h:(0,T) \to (0,\infty)$ is continuous,
nondecreasing and $h(t)\to 0$ as $t \to 0.$ The $h-$Hausdorff content of $E\subset \Bbb R$ is defined as
$$ M_h(E)=\inf\{\,\sum h(|G_i|): E \subset \cup G_i\}$$
and the $h-$Hausdorff measure of $E$ is
$$ \Lambda_h(E)=\liminf_{\delta\to 0}\{\,\sum h(|G_i|): E \subset \cup G_i, |G_i|\leq \delta\}.$$

Here we consider at most countable coverings of $E$ by intervals (open or closed).

It is easily seen that $ M_h(E)=0$ if and only if $ \Lambda_h(E)=0.$
We write $h_1 \prec h_2$ if $h_1(t)=o(h_2(t))$ as $t\to 0.$ Let
$h_1 \approx h_2$ if $C^{-1}h_1(t)\leq  h_2 \leq C h_1(t)$ for some constant $C\geq 1$ and $0<t\leq t_0<T.$
We will denote by $h_0$ the function $h_0(t) = (\log \frac{1}{t})^{-1}$ with $0<t<1,$ which defines the
logarithmic measure of sets.

A set $E$ is called {\it dimensional} if there is at least one dimension function $h$ that makes $E$
an $h-$set, that is $0 < \Lambda_h(E)<\infty.$ In our case, the set $K(\gamma)$ is dimensional. In \cite{jat}, following
Nevanlinna [Nev], the corresponding dimension function was presented. Let $\eta(\delta_k)=k$ for $k\in \mathbb{Z}_+$ with $\delta_0:= 1$ and $\eta(t)= k+ \log{\frac{\delta_k}{t}}/ \log{\frac{\delta_k}{\delta_{k+1}}}$
for $\delta_{k+1}<t< \delta_{k}.$ Then $h(t):=2^{-\eta(t)}$ for $0<t\leq 1.$ Clearly, $h(\delta_k)=2^{-k}.$

\begin{proposition}
Let $\gamma$ satisfy \eqref{ll} and $h$ be defined as above.
 Then  $ \Lambda_h(K(\gamma))=1$.
\end{proposition}
\begin{proof} Take $t=C_0 \,\delta_k,$ where $C_0$ is given in \eqref{delta}. Then
$\delta_k <t=C_0\,\gamma_k \,\delta_{k-1} <\delta_{k-1}$ for large enough $k$. Here,
$\eta(t)=k-\log C_0 /\log (1/\gamma_k)$ and $h(t)=2^{-k} \,a_k$ with
$a_k:=\exp \frac{\log C_0 \cdot \log 2}{\log (1/\gamma_k)}.$ Since $\gamma_k \to 0,$ given
$\varepsilon>0,$ there is $k_0$ such that $a_k<1+\varepsilon$ for $k\geq k_0.$ From \eqref{delta}
it follows that $1=2^k\,h(\delta_k)< \sum_{j=1}^{2^k} h(l_{j,k})< 2^k\,h(t)< 1+\varepsilon$
provided that $k\geq k_0.$ Of course, $\Lambda_h(K(\gamma))\leq \sum_{j=1}^{2^k} h(l_{j,k})$ for each $k.$
Since $\varepsilon$ is arbitrary, we get $\Lambda_h(K(\gamma))\leq 1.$

Let us show that $\Lambda_h(K(\gamma))\geq 1.$ Fix $\varepsilon>0$ and choose $k_0$ such that
 \begin{equation}\label{eps}
\varepsilon\,\log 1/\gamma_k >- \log 2 \cdot \log(1-4\gamma_k)\,\,\,\mbox{for}\,\,\,k\geq k_0.
\end{equation}
This can be done as $\gamma_k \to 0.$ Take any open covering $\cup G_i$ of $K(\gamma).$
Given $\varepsilon,$ we can consider only coverings
with $ |G_i|< \delta_{k_0}$ for each $i.$ We choose a finite subcover $\cup_{i=1}^N G_i$ of $K(\gamma).$

Fix $i\leq N$ and $k$ with $ \delta_{k+1}<|G_i|\leq \delta_{k}.$ Recall that, for each $j\leq 2^k,$ we have
$h_{j,k}> (1-4 \gamma_{k+1})l_{j,k}.$ Therefore the distance between any two basic intervals from
$E_{k+1}$ exceeds $(1-4 \gamma_{k+1})\delta_{k}.$ If $ |G_i|<(1-4 \gamma_{k+1})\delta_{k}$ then
  $G_i$ can intersect at most one interval from $E_{k+1}$. In this case we can consider only
 $ |G_i|\leq \max_{1\leq j \leq 2^{j+1}}l_{j,k+1}\leq C_0 \delta_{k+1},$ by \eqref{delta}.
 Thus there are two possibilities: $ \delta_{k+1}<|G_i|\leq C_0 \delta_{k+1}$ or  $(1-4 \gamma_{k+1})\delta_{k}< |G_i|\leq \delta_{k}.$

In the first case we have $ h(|G_i|)>2^{-k-1}.$ Here, $G_i$ intersects at most one interval from $E_{k+1}$ and, by construction, at most $2^{m-k-1}$ interval from $E_{m}$ for $m>k.$
In turn, in the latter case, $ h(|G_i|)>2^{-k}(1- \varepsilon).$ Indeed, here,
$\eta(|G_i|) < k-\log(1-4 \gamma_{k+1}) /\log (1/\gamma_{k+1})$ and $ h(|G_i|)>2^{-k}\,a,$
where $a=\exp \frac{\log(1-4 \gamma_{k+1}) \cdot \log 2}{\log (1/\gamma_{k+1})}> (1- \varepsilon),$ by \eqref{eps}.
Now $G_i$ intersects at most two interval from $E_{k+1}$ and at most $2^{m-k}$ interval from $E_{m}.$

Let us choose $m$ so large that each basic interval from $E_{m}$ belongs to some $G_i,$ perhaps not to unique.
We decompose all intervals from $E_{m}$ into two groups corresponding to the cases considered above.
 Counting intervals gives $2^m \leq \sum_i'2^{m-k-1} +  \sum_i''2^{m-k}<2^m [\sum_i'h(|G_i|)+
 \sum_i''h(|G_i|)(1- \varepsilon)^{-1}].$ From this we see that $\sum_i h(|G_i|) > 1- \varepsilon,$ which
 is the desired conclusion, as $\varepsilon$ and $(G_i)$ here are arbitrary.
\end{proof}

The same reasoning applies to a part of $K(\gamma)$ on each basic interval.
\begin{corollary}
Let $\gamma$ and $h$ be as in Proposition above, $k \in \Bbb N, 1\leq j\leq 2^k.$ Then
$ \Lambda_h(K(\gamma)\cap I_{j,k})=2^{-k}$.
\end{corollary}

Suppose, in addition, that $K(\gamma)$ is not polar. Then, by Corollary 3.2 in \cite{jat},
$\mu_{K(\gamma)}(I_{j,k})= 2^{-k}$, so the values of $\mu_{K(\gamma)}$ and the
restriction of $\Lambda_h$ on $K(\gamma)$ coincide on each basic interval. From here, by Lemma 3.3 in \cite{jat},
these measures are equal on  $K(\gamma)$.
Thus, a non-polar set $K(\gamma)$ satisfying  \eqref{ll} is indeed {\it equilibrium} Cantor-type set
if we accept for definition of this concept the condition $\mu_{K}=\Lambda_h\vert_{K(\gamma)},$
which is  more natural than the definition suggested in in [WE], Section 6.\\

We recall that there is no complete characterization of polarity of compact sets in terms of
Hausdorff measures, see e.g Chapter V in \cite{nev}. On the one hand, a set is polar if its logarithmic measure is finite.
This defines a zone $Z_{pol}$ in the scale of growth rate of dimension functions consisting of
$h$ with $\liminf_{t \to 0} h(t)/h_0(t) >0.$ If $h\in Z_{pol}$ and  $\Lambda_h(K) <\infty$ then $Cap(K)=0.$
On the other hand, functions with $\int_0 h(t)/t\,\, dt <\infty$ form a nonpolar zone $Z_{np}:$ if
 $h\in Z_{np}$ and  $\Lambda_h(K) >0$ then $Cap(K)>0.$ But, by Ursell \cite{ursel}, the remainder makes up
a zone  $Z_u$ of uncertainty. One can take two functions in this zone with $h_2 \prec h_1$ and
sets $K_1, K_2$, where $K_j$ is a $h_j-$set, such that $K_2$ is polar, $K_1$ is not, though
in the sense of Hausdorff measure the set $K_2$ is larger than $K_1$. Indeed, $\Lambda_{h_2}(K_2)>0,$ but
$\Lambda_{h_2}(K_1)=0$ or $\Lambda_{h_1}(K_2)=\infty,$ but $\Lambda_{h_1}(K_1)<\infty.$\\

Let us show that a similar circumstance is valid with the extension property.

\begin{proposition}
There are two dimension functions $h_2 \prec h_1$ and two sets $K_1, K_2$, where $K_j$ is an $h_j-$set
for $j\in \{1,2\}$, such that the smaller set $K_1$ has the extension property, whereas the larger set $K_2$
does not have.
\end{proposition}
\begin{proof}
Take $K_1$ from Example 1. Let us show that the corresponding function $h_1=2^{-\eta_1}$ is equivalent to $h_0.$
It is enough to find $C>0$ such that $\eta_0(t)-C \leq \eta_1(t) \leq \eta_0(t)+C$ for small $t.$ Here,
$ \eta_0(t)=(\log \log 1/t) /\log 2,$ so $h_0(t)=2^{-\eta_0(t)}.$ For the set $K_1$ we have
$\delta_k=\exp(-2^{k+1} B)$ and  $\eta_0(\delta_k)=k+\log 2B/\log 2.$ If $\delta_{k+1}<t\leq \delta_k$
for some $k,$ then $k\leq \eta_1(t)<k+1$ and $k+ \log 2B/\log 2\leq \eta_0(t)<k+1+\log 2B/\log 2,$ which
gives $h_1 \approx h_0.$

In turn, let $K_2$ be as in Example 2 with  $A_j=2^{k_j}\,2^{-j}$ and $\varepsilon_j=\exp(-A_j+A_{j+1})$
for $j\geq 2.$ Here we suppose that $(k_j)_{j=1}^{\infty}$ satisfies $2^{-k_j}2^j\,k_{j+1}^2\to 0$
as $j\to \infty.$ Then \eqref{irr} and \eqref{sum} are valid, which, as in Example 2, gives the lack
of the extension property for  $K_2$. Let us show that $h_2 \prec h_0.$ It is enough to check that
 $\eta_2(t)-\eta_0(t)\to \infty$ as $t\to 0.$  Let $\delta_{k} < t \leq \delta_{k-1}$ with $k_j\leq k < k_{j+1}$
 for large enough $j.$ Then $ \log 1/\delta_k = 2\log((k+5)!/5!) +A_j<2\,A_j$ and
 $\eta_0(t) < \eta_0(\delta_{k})<k_j+1-j.$ On the other hand, $\eta_2(t)\geq \eta_2(\delta_{k-1})= k-1 \geq k_j-1.$ Therefore, $\eta_2(t)-\eta_0(t)>j-2,$ which completes the proof. \end{proof}

One can suppose that, for the considered family of sets, the scale of growth rate of dimension functions
can be decomposed as above into three zones. If  $K(\gamma)$ is an $h-$set for a function $h$ with moderate
growth then the set has $EP$. If the corresponding function $h$ is large enough, then $EP$ fails. Proposition above
shows that the zone of uncertainty here is not empty.\\

We see that $h=h_0$ is not the largest function which allows $EP$ for $h-$sets $K(\gamma)$.
If, as in the regular case, we take $B_k \nearrow \infty $ of subexponential growth, then
$\delta_{k}=\exp(-2^{k+1} B_k)$ and $h_0(\delta_{k})=2^{-k-1} B_k^{-1},$ which is essentially
smaller than $h(\delta_{k})=2^{-k}$ for the corresponding function $h.$\\

{\bf Example 3}. Let $\log_{(m)} t$ denote the $m-$th iteration $\log \cdots \log t$ for large enough $t$.
The sequence $B_k=\exp(k/\log_{(m)} k)$ has subexponential growth. Then the corresponding
sequence $(\gamma_k)_{k=1}^{\infty}$ satisfies \eqref{ll}, as for large $k$ we have
$\gamma_k=\delta_{k}/\delta_{k-1}< \exp(-2^k B_k)< \exp(-2^k)$ and for the previous $k$ we can take $\gamma_k=1/32$.
By Theorem  \ref{sg}, the set $K(\gamma)$ has $EP$. Let us find a dimension function $h$ that
corresponds to this set. We will search it in the form $h(t)= h_0^{\alpha(t)}(t).$ Let $t=\delta_{k}.$
Then $\log 1/t=2^{k+1} B_k,$ so $k \sim (\log\log 1/t) /\log 2.$ On the other hand,
$h(t)=2^{-k}=(2^{k+1} B_k)^{-\alpha(t)},$ which gives $\alpha(t) \sim 1-(\log 2 \cdot \log_{(m)}k)^{-1}
\sim 1-(\log 2 \cdot \log_{(m+2)}1/t)^{-1}.$ Clearly, $h\succ h_0.$ \\

The next Proposition generalizes Example 3. We restrict our attention to strictly increasing functions $h$
of the form $h= h_0^{\alpha},$ where $\alpha$ is a monotone function on $[0,t_0]$.
Since in the next section we shall be interested in considering of dimension functions exceeding $h_0,$
let us suppose that $\alpha(t)\leq 1.$ Then $h \succ t^{\sigma}$ for each fixed $\sigma>0.$

 In addition we assume that asymptotically
\begin{equation}\label{hhh}
h(t)\leq 2\,h(t^2),
\end{equation}
which is valid for typical dimension functions corresponding to the cases\\
$a) \,\, \alpha(t)=\alpha_0\in(0,1],$\\
$b)\,\, \alpha(t)=\alpha_0+ \varepsilon(t)$ with $\alpha_0\in[0,1),$\\
$c)\,\, \alpha(t)=1-\varepsilon(t).$

Here,
\begin{equation}\label{e}
\varepsilon(t)\searrow 0 \,\,\,\,\,\mbox{with}\,\,\,\,\,
\varepsilon(t)\,\log\log 1/t \nearrow \infty \,\,\,\mbox{as}\,\,\,t\searrow 0,
\end{equation}
 since for slowly increasing $\varepsilon$ we get $h^{\alpha_0\pm \varepsilon} \approx h_0^{\alpha_0}.$

 By  \eqref{hhh}, for the inverse function $h^{-1},$ we have $h^{-1}(\tau)\leq (h^{-1}(2\tau))^2$
and $h^{-1}\prec \tau^M$ for $M$ given beforehand. From this,
$\gamma_k =h^{-1}(2^{-k})/ h^{-1}(2^{-k+1})$ defines a sequence satisfying \eqref{ll}.
We denote the corresponding set by $K^{\alpha}(\gamma).$ Our aim is to check $EP$ for this set
provided regularity of the sequence $B_k=2^{-k-1}\log(1/h^{-1}(2^{-k})).$ We see at once that
$B_k$ increases. In its turn, $\beta_k \searrow 0$ if $\alpha_0=1$ in the case $(a)$,
$\beta_k \nearrow 1/\alpha_0 -1$ in $(b)$ and in $(a)$ with $\alpha_0<1.$  Concerning $(c)$,
the monotonicity of $\beta_k$ requires additional rather technical restrictions on $\varepsilon$.
At least for $\varepsilon(t) =\varepsilon_m(t):=(\log_{(m)}1/t)^{-1}$ we have
$\beta_k \searrow 0.$ Here, $m\geq 3,$ as $h\approx h_0$ for $m\in \{1,2\}$. 

\begin{proposition}
Let $K^{\alpha}(\gamma)$ be defined by a function $h$, as above, with a regular sequence $(B_k)_{k=1}^{\infty}.$
Then $K^{\alpha}(\gamma)$ has the extension property if and only if
$$\left(\log\frac{1}{h^{-1}(2^{-k})}\right)^{1/k} \to 2 \,\,\,\mbox{as} \,\,k\to \infty.$$
\end{proposition}
\begin{proof}
 Let us find $h^{-1}$ for the case $\alpha(t)=1-\varepsilon(t).$ If $h(t)=\tau$ then
 $[1-\varepsilon(t)]\log\log 1/t=\log 1/\tau.$ Let us define a function $\delta$ by
the condition $\log\log 1/t=[1+\delta(\tau)]\log 1/\tau.$ Then  $[1-\varepsilon(t)][1+\delta(\tau)]=1,$
so $\delta(\tau)\searrow 0$ as $\tau\searrow 0.$ Then $t=h^{-1}(\tau)=\exp[-(1/\tau)^{1+\delta(\tau)}]$
and $\log(1 /h^{-1}(2^{-k}))=2^{k(1+\delta(2^{-k}))}.$ The $k-$th root of this expression tends to 2.
On the other hand, $(B_k)_{k=1}^{\infty}$ here has subexponential growth as
$\beta_k= (\delta(2^{-k})-1/k) \log 2 \to 0.$ By Theorem \ref{sg}, $K^{\alpha}(\gamma)$ has $EP.$

Similarly, if $\alpha(t)=\alpha_0+ \varepsilon(t)$ with $0<\alpha_0<1$ then
$h^{-1}(\tau)=\exp[-(1/\tau)^{1/\alpha_0-\delta(\tau)}].$ Here,
$(\log(1 /h^{-1}(2^{-k})))^{1/k} =2^{(1/\alpha_0-\delta(2^{-k}))}  \nrightarrow 2$
and $\beta_k  \nrightarrow 0,$ there is no $EP.$  In the case $(a),$ the function $\delta$ vanishes.

Lastly, $\alpha_0=0$ in $(b)$ gives $h^{-1}(\tau)=\exp[-(1/\tau)^{\Delta(\tau)}]$
with $\Delta(\tau)\nearrow \infty$ as $\tau\searrow 0.$ Here,
$(\log(1 /h^{-1}(2^{-k})))^{1/k} \to \infty$ and $ \beta_k \to \infty.$
\end{proof}

\section{Extension Property and densities of Hausdorff contents}

To decide whether a set $K$ has $EP,$ we have to consider a local structure of the most rarefied parts of $K.$
Obviously, such global characteristics as
Hausdorff measures or  Hausdorff contents cannot be applied in general for this aim.
Instead, one can suggest to describe $EP$ in terms of lower densities of $M_h$ or related functions.
Given a dimension function $h,$ a compact set $K,\,x\in K$ and $r>0,$  let
$\varphi_{h,K}(x,r):=M_h(K\cap B(x,r))$ and
$ \, \varphi_{h,K}(r):=\inf_{x\in K}\varphi_{h,K}(x,r),$
where $B(x,r)=[x-r,x+r].$ One can suppose that $K$ has  $EP$ if and only if the corresponding
function $ \varphi_{h,K}$ is not very small, in a sense, as $r \to 0.$ Essentially, this is similar
to analysis of the lower  density of the Hausdorff content, which can be defined as
$\phi_h(K):=\liminf_{r\to 0} \inf_{x\in K} \frac {M_h(K\cap B(x,r))}{M_h(B(x,r))}.$ Indeed, $M_h(B(x,r))=h(2r)$
for  $h$ with $h(t) \succ t$ and the expression above is
$\liminf_{r\to 0} \frac {\varphi_{h,K}(r)}{h(2r)}.$

In order to distinguish $EP$ by means of $\phi_h$, we have to consider
large enough dimension functions $h.$ Indeed, if for some $h_1$ with $h_1 \succ h$ there exists $h_1-$set $K_1$
with $EP$, then $h$ cannot be used for this aim, because $\Lambda_h(K_1)=0$ implies  $M_h(K_1)=0$ and the corresponding density vanishes contrary to our expectations. Therefore, we can consider only functions
exceeding $h_0.$

We remark that $\Lambda_h-$analogs of $\varphi_{h,K}$ or  $\phi_h$ cannot be applied in general for
distinguishing $EP,$ since for fat sets ($K=\overline{Int(K)}$) we have $\Lambda_h(K\cap B(x,r))=\infty$
provided $h(t) \succ t.$\\

Interestingly, it turns out that the lower  density $\phi_h$ can be used to characterize $EP$
for the family of compact sets considered in \cite{Mich}. \\

 {\bf Example 4}. Given two sequences $b_k\searrow 0$ (for brevity,
we take $b_k=e^{-k}$) and $Q_k \nearrow$ with $Q_k \geq 2,$ let $K=\{0\}\cup \bigcup_{k=1}^{\infty} I_k,$ where
$I_k=[a_k,b_k], |I_k|=b_k^{Q_k}.$ In what follows we will consider two cases: $ Q_k\leq Q$ with some $Q$ and
$Q_k\nearrow \infty$ with $Q_k<\log k$ for large $k$. By Theorem 4 in \cite{Mich}, $K$ has the extension property
in the first case and does not have it for unbounded $(Q_k)$.\\

In the next lemma we consider concave dimension functions $h= h_0^{\alpha}$ for the cases $(a), (b),$ as above,
and for more general\\
$c\,')\,\, \alpha(t)=\alpha_0-\varepsilon(t)$ with $\alpha_0\in(0,1].$\\
 We suppose now that $\varepsilon$ is a monotone differentiable function on $[0,t_0]$ with $0<\varepsilon(t)<1-\alpha_0$ in  $(b)$ and  $0<\varepsilon(t)<\alpha_0/2$ in  $(c\,').$
 As before, we assume \eqref{e}. A direct computation shows that
\begin{equation}\label{deriv}
h'(t)< h(t)\,h_0(t)\,\alpha(t)/t\,\,\,\mbox{for the cases}\,(a), (b)\,\,\mbox{and}\,\,\,\,h'(t)< h(t)\,h_0(t)/t\,\,\,\mbox{for}\,\, (c\,').
\end{equation}

\begin{lemma} \label{M} Suppose intervals $I_k$ are given as in Example 4 and $n$ is large enough.
Then $M_h(\cup_{k=n}^{\infty}I_k)=h(b_n)$. This
means that the covering of the set $\cup_{k=n}^{\infty}I_k$ by the interval $[0,b_n]$ is optimal in the sense
of definition of $M_h.$
\end{lemma}

\begin{proof} Let us fix any open covering of $K,$ choose its finite subcovering $\cup_{i=1}^M G_i$ and enumerate
sets $G_i$ from left to right. We can suppose that $G_1$ covers $\cup_{k=N}^{\infty}I_k$ for some $N\geq n.$
Indeed, if $G_1$ covers as well some part of $I_{N-1},$ then other part of $I_{N-1}$ is covered by $G_2.$
In this case, association of $G_1$ and $G_1$ into one interval will give better covering, since
$h(b)\leq h(x)+h(b-x)$ for $0\leq x \leq b,$ by concavity of $h$. For the same reason, we suppose that each
$G_i$ covers entire number of $I_k.$ After this we reduce each $G_i$ to the minimal closed interval $F_i$
containing the same intervals  $I_k.$ Thus, $F_1=[0,b_N]$ and $F_2=[a_{N-1}, b_q]$ with some $N-1\leq q\leq n.$
Our aim is to show that
\begin{equation}\label{optim}
h(b_q) < h(b_N)+h(b_q-a_{N-1}),
\end{equation}
so replacing $F_1 \cup F_2$ with $[0,b_q]$ is preferable. We use the mean value theorem and the decrease
of $h'.$ Note that $h(b_k)=k^{-\alpha(b_k)}.$\\

Consider first the value $q=N-1.$ We will show $h(b_{N-1}) - h(b_N)< h(|I_{N-1}|).$

In the cases $(a), (b),$ by \eqref{deriv}, $LHS< h'(b_N)e^{-N}(e-1)<N^{-1-\alpha(b_N)} \alpha(b_N) (e-1).$
On the other hand, $h(|I_{N-1}|)=[Q_{N-1}(N-1)]^{-\alpha(|I_{N-1}|)}.$ Here, $\alpha(|I_{N-1}|)<\alpha(b_N),$
so we reduce the desired inequality to $(Q_{N-1}/N)^{\alpha(b_N)}\alpha(b_N) (e-1)<1.$ It is valid, since
for $\alpha_0>0$ the first term on the left goes to zero, whereas for $\alpha_0=0$ in $(b)$ we have
$\alpha(b_N)=\varepsilon(b_N) \to 0$ as $N\to \infty.$

Similarly, in the case $(c\,')$ the inequality
$[Q_{N-1}(N-1)]^{\alpha_0- \varepsilon(|I_{N-1}|)}(e-1)<N^{1+\alpha_0- \varepsilon(b_N)}$ is valid,
as is easy to check.

Suppose now that $q\leq N-2.$ We write \eqref{optim} as  $h(b_q) - h(b_q-a_{N-1})< h(b_N).$

Here, in all cases, by \eqref{deriv}, $LHS< h'(b_q-a_{N-1})\,a_{N-1}<h(b_q) h_0(b_q) \frac{a_{N-1}}{b_q-a_{N-1}},$
where the last fraction does not exceed $\frac{b_{N-1}}{b_q-b_{N-1}}.$ On the other hand,
$ h(b_N)\geq N^{-1}$ as $\alpha(b_N)\leq 1.$ Hence it is enough to show that
$N< (e^{N-q-1}-1)\, q^{1+\alpha(b_q)}.$ We neglect $\alpha(b_q)$ and notice that
$(e^{N-q-1}-1)\, q\geq (e-1) (N-2),$ which completes the proof of \eqref{optim}.

Continuing in this manner, we see that $h(b_n) \leq \sum_{i=1}^M h(|F_i|).$
\end{proof}

\begin{corollary}\label{cor2}
Suppose $b_{n+1}\leq r \leq b_n-b_{n+1}.$ Then $\varphi_{h,K}(r)=h(|I_n|).$
\end{corollary}
\begin{proof}
Clearly, $\varphi_{h,K}(x,r)=h(|I_n|)$ for each $x\in I_n.$ If $x\in K\cap [0,b_{n+1}]$ then
$B(x,r)$ covers all intervals $I_k$ with $k\geq n+1.$ By Lemma, $\varphi_{h,K}(x,r)=h(b_{n+1})>h(|I_n|).$
Of course, for $x\in I_k$ with $k<n$ the value $\varphi_{h,K}(x,r)$ also exceed $h(|I_n|)$.
\end{proof}

{\bf Remark.} The covering of two (or small number of) intervals $I_k$ by one interval is not optimal,
since $M_h(I_k\cup I_{k+1})=h(|I_k|)+h(|I_{k+1}|)< h(b_k-a_{k+1})$.
\vspace{0.3cm}

We proceed to characterize $EP$ for given compact sets in terms of lower densities $\phi_h$ for $h= h_0^{\alpha},$ where
\begin{equation}\label{a}
\alpha(t)=\alpha_0\in(0,1]\,\,\mbox{or}\,\,\alpha(t)=\alpha_0\pm \varepsilon_m(t)
\end{equation}
 with $0<\alpha_0<1$ and $\varepsilon_m(t)=(\log_{(m)}1/t)^{-1}$ for $m>2$, so \eqref{e} is valid.

\begin{proposition}\label{phi1}
Let $K$ be from the family of compact sets given in Example 4 and $h$ be as above.
Then $K$ has the extension property if and only if   $\phi_h(K)>0.$
\end{proposition}
\begin{proof}
Suppose first that $ Q_k\leq Q$ with some $Q,$ so $K$ has $EP$. We aim to show
$\underline{\lim}_{r\to 0} \frac {\varphi_{h,K}(r)}{h(2r)}>0.$ Let $e^{-k-1}\leq r < e^{-k}$
for some $k.$
Then, as $\varphi_{h,K}$ increases, $\varphi_{h,K}(r)\geq \varphi_{h,K}(e^{-k-1}),$
which is $h(|I_k|)=(k\cdot Q_k)^{-\alpha(|I_k|)},$ by Corollary \ref{cor2}.
On the other hand, $h(2r)<h(2e^{-k})=(k-\log 2)^{-\alpha(2e^{-k})}.$
Therefore,
$$\varphi_{h,K}(r)/h(2r)>Q_k^{-\alpha(|I_k|)}\,\,k^{\alpha(2e^{-k})-\alpha(|I_k|)}\,\,
(1-\log 2/k)^{-\alpha(2e^{-k})}.$$
 The first term on the right converges to  $Q^{-\alpha_0}$ as $k\to \infty$.
 The second and the third terms converge to 1.
Hence, $\phi_h(K) \geq Q^{-\alpha_0}.$ Besides, this value is achieved in the case
$Q_k=Q$ by the sequence $r_k=b_k-b_{k+1}.$ Thus, $\phi_h(K)>0.$ In addition,
for given $\sigma>0$ we have a compact sets $K$ with $EP$ such that $0<\phi_h(K) < \sigma.$

Similar arguments apply to the case $Q_k\nearrow \infty,$ when $K$ does not have $EP.$  Here,
$\phi_h(K)\leq \lim_k \varphi_{h,K}(r_k)/h(2r_k)$ for $r_k$ as above.
By Corollary \ref{cor2}, $\varphi_{h,K}(r_k)=h(|I_k|).$ Also, $h(2r_k)>h(e^{-k}).$
Hence, $\varphi_{h,K}(r_k)/h(2r_k)< Q_k^{-\alpha_0/2}\,k^{\alpha(e^{-k})-\alpha(|I_k|)},$
which converges to 0 as $k$ increases.
\end{proof}

{\bf Remark.} For this family of sets, the extension property can be characterized as well in
terms of the Lebesgue linear measure $\lambda.$ Let $\lambda(r):=\inf_{x\in K} \lambda(K\cap [x-r,x+r]).$
Then $K$ has the extension property if and only if
$\liminf_{r\to 0} \lambda(r) \cdot r^{-Q}>0$ for some $Q$.\\

Nevertheless, at least for dimension functions $h=h_0^{\alpha}$ with $\alpha$ as in \eqref{a}, there is no general characterization of $EP$ in terms of lower densities $\phi_h.$ In view of Example 3 and the discussion in the beginning of the section, the value $\alpha_0=1$ can be omitted from consideration

We treat now regular sets $K(\gamma)$ with $\delta_k=\exp(-b^k).$ Here, $B_k=2^{-1} (b/2)^k$.
By Theorem \ref{sg}, $K(\gamma)$ has  $EP$ if $b=2$ and does not have it for $b > 2.$

\begin{lemma} \label{h}
For each constants $C\geq 1$ and $h$, as above, there is $b>2$ such that
$ h(C\delta_k)< 2\,h(\delta_{k+1})$ for large enough $k$. This inequality is valid also for $b=2.$
\end{lemma}
\begin{proof}
In all cases we have $h(\delta_k)=b^{-k \cdot \alpha(\delta_k)}$ and the desired inequality has the form
\begin{equation}\label{b}
b^{(k+1)\alpha(\delta_{k+1})}<2\, (b^k-\log C)^{\alpha(C\delta_k)}.
\end{equation}

Suppose $\alpha\equiv \alpha_0.$ Then \eqref{b} is valid as $b^{\alpha_0}<2\,(1-b^{-k}\log C)$
for large $k$ and $b=2+\sigma$ with small enough $\sigma$. All the more it is valid for $b=2.$

The same reasoning applies to the case $\alpha=\alpha_0+ \varepsilon(t)$ with $\varepsilon \nearrow$
 as $\varepsilon(\delta_{k+1})< \varepsilon(C\delta_k).$

In the last case  $\alpha(t)=\alpha_0 - \varepsilon_m(t)$ we use the following simple inequality
$$\log_{(m)}(Cx)-\log_{(m)}(x)< \log C \cdot[\log x\,\log_{(2)}(x) \cdots \log_{(m-1)}(x)]^{-1},$$
which is valid for all $x$ from the domain of definition of $\log_{(m)}$. From this we have
$k\cdot [ \varepsilon(C\delta_k)-\varepsilon(\delta_{k+1})] \to 0$ as $k\to \infty$ and
 \eqref{b} can be treated as in the first case.
\end{proof}

\begin{corollary}\label{cor3} Let $k$ be large enough.Then the covering of each basic interval $I_{j,k}$ of $K(\gamma)$ by one interval is better (in the sense
of definition of $M_h$) than covering by two adjacent subintervals.
\end{corollary}

Indeed, by \eqref{delta}, $h(l_{j,k})< h(C_0\delta_k)<  2\,h(\delta_{k+1})< h(l_{2j-1,k+1})+ h(l_{2j,k+1}).$\\

{\bf Remark.} It is essential that coverings of a whole basic interval are considered. For example,
for the set $I_{1,k}\cup I_{3,k+1}$ we have $h(l_{1,k})+ h(l_{3,k+1})<h(b_{3,k+1}),$ which corresponds
to the covering of the set by one interval.

\begin{proposition}\label{phi2}
Let $h=h_0^{\alpha}$ with $\alpha$ as in \eqref{a} and $K(\gamma)$ be defined by  $\delta_k=\exp(-b^k).$
Then $\phi_h(K(\gamma))=b^{-\alpha_0}$. 
\end{proposition}
\begin{proof}
For brevity, we denote here $K(\gamma)$ by $K.$ Fix $x \in K.$ Let $x\in I_{j,k}\subset I_{i,k-1} $ and $C_0\,\delta_k\leq r \leq 7/8 \cdot \delta_{k-1}.$ Then, by \eqref{delta}, $l_{j,k}\leq r <h_{i,k-1}$ and
$K \cap [x-r, x+r]=K \cap I_{j,k}.$ Arguing as in Lemma \ref{M}, by Lemma \ref{h}, we get
$\varphi_{h,K}(x,r)=h(l_{j,k}).$ Therefore, by monotonicity, $h(\delta_k)<\varphi_{h,K}(x,r)<h(C_0 \delta_k)$
for each $x\in K$.

We proceed to estimate $\phi_h(K)$ from both sides. Suppose that $C_0\,\delta_k\leq r \leq C_0 \cdot \delta_{k-1}$
for some $k.$
Then $h(\delta_k)<\varphi_{h,K}(r)<h(C_0 \delta_{k-1})$ and
$$ \frac{h(\delta_k)}{h(2C_0\delta_{k-1})} < \frac{\varphi_{h,K}(r)}{h(2r)}< \frac{h(C_0\delta_{k-1})}{h(2C_0\delta_k)}.$$
Here, $\delta_k=\delta_{k-1}^b.$ Analysis similar to that in the proof of Lemma \ref{h} shows that
the first fraction above has the limit $b^{-\alpha_0}$, whereas the last fraction tends to
$b^{\alpha_0}$ as $k \to \infty.$ Moreover, the value $b^{-\alpha_0}$ can be achieved as
$\lim_k \varphi_{h,K}(r_k)/ h(2r_k)$ for $r_k=7/8\cdot \delta_{k-1}.$
\end{proof}

Comparison of Propositions \ref{phi1} and \ref{phi2} shows that, for given dimension functions,
lower densities of Hausdorff contents cannot be used in general to characterize the extension property.

\section{Extension Property and growth of Markov's factors}

Let ${\mathcal P}_n$ denote the set of all holomorphic polynomials of degree at most $n.$
For any infinite compact set $K \subset {\Bbb C}$ we consider the sequence of {\it Markov's factors}
$$M_n(K)=\inf \{ M: \,|P'|_{0,K} \leq M \,|P|_{0,K}, \,\,P\in {\mathcal P}_n \}$$
for  $n\in {\Bbb N}.$
We see that $M_n(K)$ is the norm of the operator of differentiation in the space
$({\mathcal P}_n, |\cdot|_{0,K}).$ We say that a set $K$ is {\it Markov} if the sequence $(M_n(K))$
is of polynomial growth. This class of sets is of interest to us, since, by  W.Ple\'sniak \cite{ples},
any Markov set has $EP.$ On the other hand, there exist non-Markov compact sets with $EP$ (\cite{G96}, \cite{AG}).
We guess that there is some extremal growth rate $(m_n)_{n=1}^{\infty}$ with the property:
if, for some compact set $K$, $M_n(K)/m_n \to \infty$ as $n\to \infty$ then $K$ does not have $EP.$
The next proposition asserts that here, as above, there is a zone of uncertainty, in which
growth rate of  Markov's factors is not related with $EP.$ In this sense, it is an analog of
Proposition 8.3.

\begin{proposition}\label{mark}
There are two sets $K_1$ with $EP$ and $K_2$ without it, such that $M_n(K_1)$ grows essentially faster than $M_n(K_2)$ as $n\to \infty$.
\end{proposition}
\begin{proof}
By Theorem 6 in [we], $M_{2^k}(K(\gamma))\sim 2/\delta_k.$ By monotonicity,
$ \delta_k^{-1}<M_n(K(\gamma))<4\,\delta_{k+1}^{-1}$ for $2^k\leq n <2^{k+1}$ with large enough $k$.
As in Proposition 8.3, we take $K_1$ from Example 1, so $\delta_k^{(1)}=\exp(-2^{k+1}B)$ with $B>1.$
 Also, we use $K_2$ from Example 2 with $A_j=2^{k_j}$. For simplicity, we fix $k_j=j^2$ that satisfies
 \eqref{irr}. Here,
 $\delta_k^{(2)}> k^{-2k}\,\varepsilon_1\varepsilon_2\cdots \varepsilon_j$ for $k_j\leq k < k_{j+1}.$
 We aim to show that $M_n(K_2) /M_n(K_1)\to 0$ as $n\to \infty.$ Let us fix large $n$ with
 $2^k\leq n < 2^{k+1}.$  For this $k$ we fix $j$ with  $k_j\leq k < k_{j+1}.$ Then
 \begin{equation}\label{mm}
M_n(K_2) /M_n(K_1)< 4\,\delta_k^{(1)}/ \delta_{k+1}^{(2)}.
\end{equation}

Suppose first that $k\leq  k_{j+1}-2.$ Then RHS of \eqref{mm} does not exceed
$4\,\exp[-2^{k+1}B + 2(k+1)\log(k+1) + A_j].$ The expression in brackets is smaller than
$2^{k_j}(1-2B)+ k_{j+1}^2,$ which is $(j+1)^4-(2B-1)\,2^{j^2},$ so it tends to $-\infty$ as $j\to \infty.$

If  $k= k_{j+1}-1$ then RHS of \eqref{mm} is smaller than
$4\,\exp[-2^{ k_{j+1}}B + 2 k_{j+1}\log k_{j+1} + A_{j+1}],$ which goes to 0, since $B>1.$
This completes the proof.
 \end{proof}

Existence of a zone of uncertainty (for the extension property) in  the scale of growth rate of  Markov's factors
implicates the problem to find boundaries of this zone.

\end{document}